\documentclass[a4paper,notitlepage,twoside,reqno,11pt]{amsart}

\usepackage{anysize}
\marginsize{3.4cm}{3.4cm}{3cm}{3cm}

\usepackage{amsmath}
\usepackage{amsthm}
\usepackage{amssymb}
\usepackage{color}
\usepackage{indentfirst}
\usepackage{graphicx}
\usepackage{bbding}

\usepackage[hypertexnames=false,bookmarksopen=true,linktocpage=true,pdfstartview={XYZ null null 1.25}]{hyperref}

\usepackage{autonum}

\begin{document}
	
\newtheorem{thm}{Theorem}[section]
\newtheorem{lem}[thm]{Lemma}
\newtheorem{cor}[thm]{Corollary}
\newtheorem{prop}[thm]{Proposition}
\newtheorem*{mainthm}{Main Theorem}
\newtheorem{thmx}{Theorem}
\renewcommand{\thethmx}{\Alph{thmx}} 
\newtheorem{corx}[thmx]{Corollary}

\theoremstyle{definition}
\newtheorem*{defi}{Definition}
\newtheorem*{exam}{Example}
\newtheorem*{rmk}{Remark}
\newtheorem*{ques}{Question}
\newtheorem*{conj}{Conjecture}
\newtheorem*{nota}{Notations}
	
\renewcommand{\theequation}{\arabic{section}.\arabic{equation}}
\makeatletter
\@addtoreset{equation}{section}
\makeatother
	
\newcommand{\ind}{\mathop{\mathrm{index}}}
\newcommand{\R}{\mathbb{R}}
\newcommand{\Z}{\mathbb{Z}}
\newcommand{\N}{\mathbb{N}}
\newcommand{\J}{\mathfrak{J}}
\newcommand{\C}{\mathbb{C}}
\newcommand{\D}{\mathbb{D}}
\newcommand{\EC}{\mathbb{\widehat{C}}}

\newcommand{\Area}{\textup{Area}}
\newcommand{\MA}{\mathcal{A}}
\newcommand{\MH}{\mathcal{H}}
\newcommand{\MM}{\mathcal{M}}
\newcommand{\MP}{\mathcal{P}}
\newcommand{\MU}{\mathcal{U}}
\newcommand{\Crit}{\textup{Crit}}
\newcommand{\ii}{\textup{i}}

\author{Xiaole He, Yingqing Xiao and Fei Yang}

\address{School of Mathematics, Hunan University, Changsha, 410082, P.R. China}
\email{hxlvcl@foxmail.com}

\address{School of Mathematics, Hunan University, Changsha, 410082, P.R. China}
\email{ouxyq@hnu.edu.cn}

\address{School of Mathematics, Nanjing University, Nanjing, 210093, P.R. China}
\email{yangfei@nju.edu.cn}
	
\title[Generalized Sierpi\'{n}ski gasket Julia sets]{Rational maps whose Julia sets are generalized Sierpi\'{n}ski gaskets}

\begin{abstract}
It has been shown that the Sierpi\'nski gasket-like sets can appear as the Julia sets of some geometrically finite rational maps. In this paper we prove that such type of Julia sets can also appear in the rational maps containing Siegel disks, Cremer points or which are infinitely renormalizable. Based on this, we prove the existence of gasket Julia sets with positive area. Moreover, we present a criterion which guarantees the existence of gasket Julia sets in some rational maps having exactly one fixed attracting or parabolic basin.
\end{abstract}

\date{\today}


\maketitle

\section{Introduction}

\subsection{Backgrounds}
The attractors generated by iterated function systems in fractal geometry and the Julia sets generated by rational maps in complex dynamics share many similar geometric and topological properties. Sierpi\'{n}ski carpets and gaskets are typical objects which were studied in these two fields.

For Sierpi\'{n}ski carpets, there exists a topological characterization due to Whyburn \cite{Why58}: A \emph{Sierpi\'{n}ski carpet} (or \emph{carpet} in short) is a subset in the sphere which is compact, connected, locally connected, has empty interior and has the property that the complementary domains are bounded by pairwise disjoint simple closed curves. All carpets are homeomorphic to the standard middle third carpet. The first example of carpet Julia sets was found by Milnor and Tan \cite[Appendix F]{Mil93}. Later, more carpet Julia sets were found in \cite[Section 5.6]{Pil94}, \cite{DLU05}, \cite{Ste06}, \cite{Ste08}, \cite{DFGJ14}, \cite{XQY14} and \cite{Yan18} etc. (see also the references therein). In particular, in the $\lambda$-parameter plane of the \textit{McMullen family}
\begin{equation}\label{equ:McM}
f_\lambda(z)=z^m+\lambda/z^n,
\end{equation}
where $m\geqslant 2$, $n\geqslant 1$ and $\lambda\in\C\setminus\{0\}$, there are infinitely many hyperbolic components called \textit{Sierpi\'{n}ski holes} whose corresponding Julia sets are Sierpi\'{n}ski carpets (see \cite{DLU05}, \cite{Dev05} and \cite{Ste06}).

However, for Sierpi\'{n}ski gaskets, there is no such topological characterization as Sierpi\'{n}ski carpets. The classical \textit{Sierpi\'{n}ski gasket} (or \textit{triangle}) is obtained by repeatedly replacing an equilateral triangle by three triangles of half the height, and this set is the attractor of the iterated function system consisting three similarities with contraction factor $1/2$ (see the left picture in Figure \ref{Fig:gasket}). The Julia set which is homeomorphic to the classical Sierpi\'{n}ski gasket was firstly found by Ushiki \cite{Ush91} and this Julia set is generated by the specific McMullen map $f_\lambda(z)=z^2+\lambda/z$ with $\lambda=-\frac{16}{27}$ (see also \cite{Kam00}, \cite{Dev04} and the right picture in Figure \ref{Fig:gasket}).

\begin{figure}[!htpb]
  \setlength{\unitlength}{1mm}
  \centering
  \includegraphics[width=0.42\textwidth]{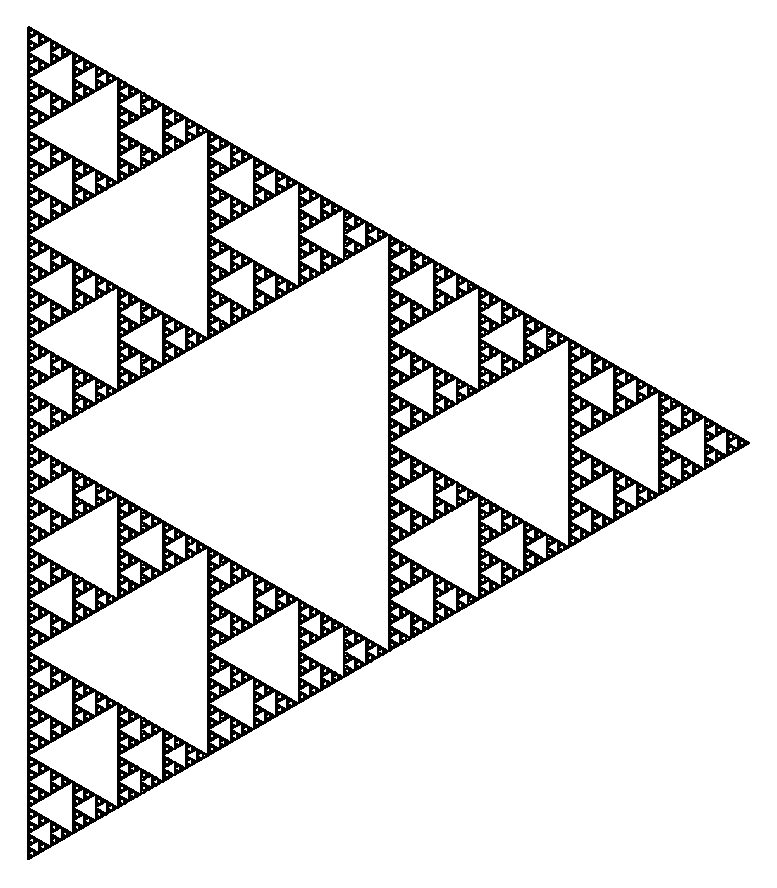} \quad
  \includegraphics[width=0.42\textwidth]{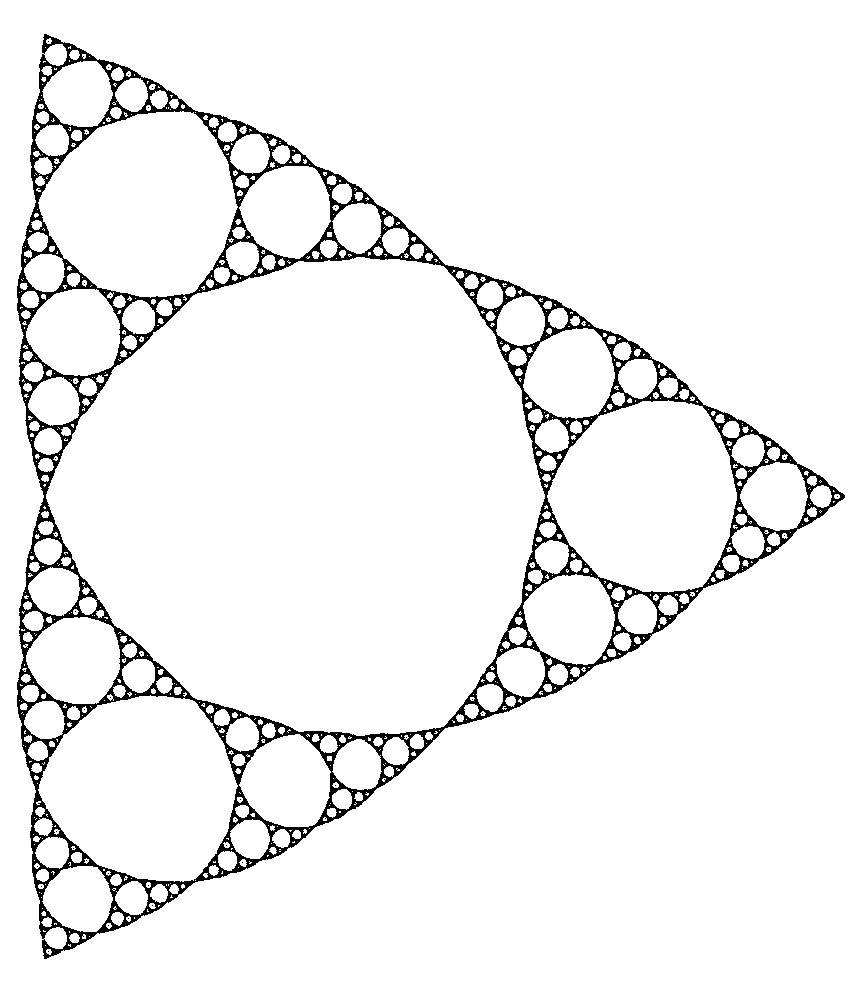}
  \caption{The classical Sierpi\'{n}ski gasket and a Sierpi\'{n}ski gasket Julia set. The picture on the right is generated by the McMullen map $f_\lambda(z)=z^2+\lambda/z$, where $\lambda=-\frac{16}{27}$ such that the critical point $-\frac{2}{3}$ has the finite forward orbit $-\frac{2}{3}\mapsto\frac{4}{3}\mapsto\frac{4}{3}$ on the boundary of the immediate super-attracting basin of the infinity. }
  \label{Fig:gasket}
\end{figure}

It turns out that it is difficult to find a post-critically finite rational map, such that it is not the iteration of Ushiki's example and that its Julia set is homeomorphic to the classical Sierpi\'{n}ski gasket. In view of this, Devaney et al. introduced the concept of \textit{generalized Sierpi\'{n}ski gaskets}, which are obtained by a similar recursive process as the classical Sierpi\'{n}ski gasket, but at each stage, $N$ homeomorphic copies of an $N$-polygon are removed ($N=3$ for the classical Sierpi\'{n}ski gasket, see \cite[Section 2.1]{DRS07}). They proved that if the McMullen family $f_\lambda$ is of \textit{Misiurewicz-Sierpi\'{n}ski} type, i.e., each critical point of $f_\lambda$ in $\C\setminus\{0\}$ is strictly preperiodic and lies on the boundary of the immediate super-attracting basin of the infinity, then the Julia set $J(f_\lambda)$ is a generalized Sierpi\'{n}ski gasket. However, unlike the Sierpi\'{n}ski carpets, except the symmetric case, each pair of these Julia sets are not only topologically distinct but also dynamically different (see also \cite{BDL06} and \cite{DHL08}).

Recently, the Sierpi\'{n}ski gasket-like sets attracted many people's interest, in the fields not only of fractal geometry and complex dynamics, but also of geometric group theory (see \cite{LLMM23}) and quasiconformal geometry (see \cite{Nta19}, \cite{LN24}, \cite{LZ25}). Before stating our main results, we first introduce the following definition.

\begin{defi}[Generalized Sierpi\'nski gaskets]\label{GSgasket}
Let $K\subset\EC$ be a locally connected continuum with empty interior that has the following properties:
    \begin{enumerate}
        \item each complementary component of $K$ is a Jordan domain;
        \item there exists $N\geqslant 1$ such that the boundaries of any two complementary components of $K$ share at most $N$ points;
        \item no three complementary components have a common boundary point; and
        \item the contact graph corresponding to $K$, obtained by assigning a vertex to each complementary component and an edge if two touch, is connected.
    \end{enumerate}
Then $K$ is called a \textit{generalized Sierpi\'nski gasket} (or \textit{gasket} in short).
\end{defi}

The definition above is inspired by \cite{LN24}, where $N=1$ is required. However, note that for the classical Sierpi\'{n}ski gasket, we have $N=3$.
It is not hard to see that the generalized Sierpi\'{n}ski gaskets introduced in \cite{DRS07} satisfy all the conditions in the definition above (see the proof of Lemma \ref{lem:crit-on-B}).
Hence the above definition can be seen as a generalization of both \cite{DRS07} and \cite{LN24}.

\subsection{Main results}

For the McMullen family $f_\lambda(z)=z^m+\lambda/z^n$, where $m\geqslant 2$, $n\geqslant 1$ and $\lambda\in\C\setminus\{0\}$, there exists a unique unbounded hyperbolic component $\mathcal{H}_0$ such that if $\lambda\in\mathcal{H}_0$, then the Julia set $J(f_\lambda)$ of $f_\lambda$ is a Cantor set (see \cite{DLU05}) and if $\lambda\in\partial \mathcal{H}_0$, then $J(f_\lambda)$ is connected (see \cite{DR13}).
As we have mentioned above, if $f_\lambda$ is a Misiurewicz-Sierpi\'{n}ski map, then $J(f_\lambda)$ is a generalized Sierpi\'{n}ski gasket \cite{DRS07}.
In the case $m=n\geqslant 3$, if $f_\lambda$ is of Misiurewicz-Sierpi\'{n}ski, then $\lambda\in\partial\mathcal{H}_0$ (see \cite[Theorem 6.5]{QRWY15}).
We first extend the result of \cite{DRS07} to a general case.

\begin{thm}\label{thm:McM-family}
If $\lambda\in\partial \mathcal{H}_0$, then the Julia set of $f_\lambda(z)=z^m+\lambda/z^n$ with $m=n\geqslant 3$ is a generalized Sierpi\'nski gasket.
\end{thm}

We guess that the result in Theorem \ref{thm:McM-family} is also true for $m\neq n$ but the proof in this paper relies on some results in \cite{QWY12} and \cite{QRWY15}, where only $m=n\geqslant 3$ was considered.
Besides critically finite maps, $\partial \mathcal{H}_0$ also contains some critically infinite and parabolic parameters (see left pictures in Figures \ref{Fig:Julia-gasket-Siegel} and \ref{Fig:Julia-gasket-McM}).

\begin{figure}[!htpb]
  \setlength{\unitlength}{1mm}
  \centering
  \includegraphics[width=0.45\textwidth]{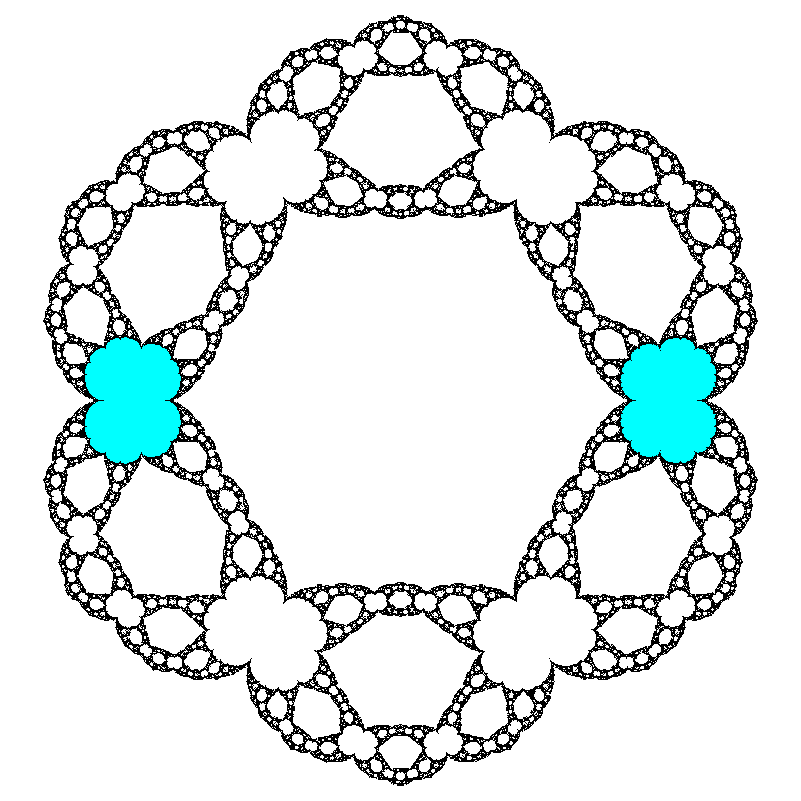}\quad
  \includegraphics[width=0.45\textwidth]{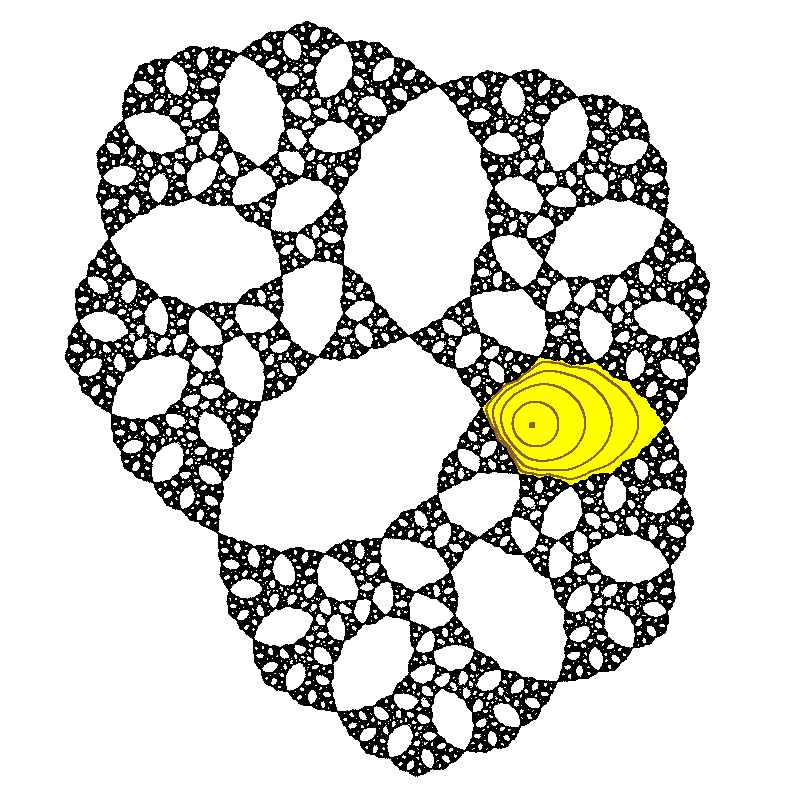}
  \caption{Left: A generalized Sierpi\'{n}ski gasket Julia set of $f_\lambda(z)=z^3+\lambda/z^3$ with parabolic points (whose immediate parabolic basins are colored cyan). Right: A generalized Sierpi\'{n}ski gasket Julia set of a quadratic rational map containing a Siegel disk (colored yellow).}
  \label{Fig:Julia-gasket-Siegel}
\end{figure}

To study the quasisymmetric geometry of Julia sets, it is important to know whether they have zero area. Bonk, Lyubich and Merenkov studied the quasisymmetric geometry of carpet Julia sets of critically finite rational maps \cite{BLM16} and later some results were extended to the semi-hyperbolic case \cite{QYZ19}. These carpet Julia sets have zero area although some other carpet Julia sets may have positive area \cite{FY20}.  Recently, Luo, Ntalampekos and Zhang studied the quasisymmetric geometry of gasket Julia sets of geometrically finite rational maps (see \cite{LN24}, \cite{LZ25}).
To the best of our knowledge, except the examples in Theorem \ref{thm:McM-family} (These Julia sets have zero area by \cite[Theorem 5.4]{QRWY15}), all the known gasket Julia sets are generated by geometrically finite rational maps and they have Hausdorff dimension strictly less than two \cite{Urb94}. In this paper we prove the following result.

\begin{thm}\label{thm:area-dim}
There exist rational maps having gasket Julia sets with positive area, and they either contain a Cremer point or are infinitely renormalizable.
\end{thm}

The existence of quadratic polynomials having Julia sets with positive area was known (see \cite{BC12}, \cite{AL22}, \cite{DL23a}), and they either contain a Cremer point, a Siegel disk or are infinitely renormalizable. We will use polynomial-like renormalization theory to ``arrange" some of these quadratic polynomial Julia sets in the gasket Julia sets of a family of quadratic rational maps such that they have positive area (see \cite{FY20} for a similar idea). Moreover, we will also show that there exist Siegel rational maps having gasket Julia sets with Hausdorff dimension strictly less than two (see the right picture in Figure \ref{Fig:Julia-gasket-Siegel}).

\medskip

In view of that the proof methods for the existence of gasket Julia sets above depend on the special families, it is meaningful to give a criterion to judge whether a given rational map has a gasket Julia set.

\begin{thm}\label{thm:criterion}
Let $f$ be a rational map satisfying the following conditions:
\begin{enumerate}
\item the Julia set $J(f)$ of $f$ is connected;
\item there exist two different Fatou components $U$ and $V$ of $f$ such that $f(U)=U$ and $f^{-1}(U) = U\cup V$;
\item if $W$ is a connected component of $f^{-1}(V)$, then $\mathrm{deg}(f|_W)>\mathrm{deg}(f|_U)$;
\item $\partial U$ contains exactly one critical point, which is simple and strictly preperiodic, and any other critical orbit of $f$ intersects with $U$.
\end{enumerate}
Then $J(f)$ is a generalized Sierpi\'nski gasket.
\end{thm}

A similar criterion as Theorem \ref{thm:criterion} was given in \cite{Yan18} for carpet Julia sets. As an application of Theorem \ref{thm:criterion}, we consider the following one-dimensional family of rational maps (see \cite[Proposition 1.4]{Yan18}):
\begin{equation}\label{equ:g-eta}
g_\eta(z)=\frac{a(z-1)^3(z-b)^3}{z^4}, \text{\quad where }
a=\frac{\eta(\eta-4)^3}{27(\eta-1)^6},~
b=\frac{\eta(1+2\eta)}{4-\eta}
\end{equation}
and $\eta\in\C\setminus\{0,1,4,-\tfrac{1}{2}\}$ is the parameter. A direct computation shows that $g_\eta$ has the critical orbits $b\overset{(3)}{\longmapsto} 0\overset{(4)}{\longmapsto}\infty\overset{(2)}{\longmapsto}\infty$, $\eta\overset{(2)}{\longmapsto} 1\overset{(3)}{\longmapsto} 0$ and a \textit{free} critical point $c=\frac{2(1+2\eta)}{\eta-4}$ with local degree $2$.
Let $U_\eta$ be the immediate super-attracting basin of $\infty$. As an application of Theorem \ref{thm:criterion}, we have the following result.

\begin{cor}\label{cor:gasket-2}
If the critical point $c$ lies on $\partial U_\eta$ and it is strictly preperiodic, then $J(g_\eta)$ is a generalized Sierpi\'{n}ski gasket.
\end{cor}

This paper is organized as following:
In Section \ref{sec:McM}, we give the proof of Theorem \ref{thm:McM-family} based on some dynamical properties of the McMullen map $f_\lambda$ when $\lambda$ lies on the boundary of unbounded hyperbolic component.
In Section \ref{sec:area-dim}, we study a family of quadratic rational maps containing a fixed bounded type Siegel disk and prove that there are parameters such that the corresponding Julia sets are generalized Sierpi\'{n}ski gaskets with positive area (Theorem \ref{thm:area-dim}).
In Section \ref{sec:criterion}, we prove Theorem \ref{thm:criterion} and use this criterion to obtain Corollary \ref{cor:gasket-2}.

\medskip
\noindent\textbf{Acknowledgements.}
The authors would like to thank the referee for careful reading and helpful comments.
This work was supported by NSFC (Grant Nos. 12471073, 12571093, 12222107) and NSF of Hunan Province (Grant No. 2025JJ50049).

\section{Gasket Julia sets in the McMullen family}\label{sec:McM}

In this section, we consider the special McMullen family
\begin{equation}\label{equ:McM-3}
f_\lambda(z)=z^n+\lambda/z^n, \text{\quad where }n\geqslant 3
\end{equation}
and $\lambda\in\C^*:=\C\setminus\{0\}$. The dynamical plane and parameter space of this family were deeply studied in \cite{QWY12} and \cite{QRWY15} (see also \cite{DLU05}, \cite{Ste06}). We shall use some results therein to prove Theorem \ref{thm:McM-family}.

\subsection{Polynomial-like renormalization}\label{subsec:poly-like}

Let $U$ and $V$ be two Jordan disks in $\C$ such that $U$ is compactly contained in $V$. The map $f:U\to V$ is called a \textit{polynomial-like mapping} of degree $d\geqslant 2$ if $f$ is a proper holomorphic surjection with degree $d$. The \emph{filled Julia set} $K(f)$ is defined as $K(f):=\bigcap_{n\geqslant 0}f^{-n}(V)$ and the \textit{Julia set} $J(f)$ is the topological boundary of $K(f)$.
Two polynomial-like mappings $f_1:U_1\to V_1$ and $f_2:U_2\to V_2$ are said to be \emph{hybrid equivalent} if there is a quasiconformal mapping $h$ defined from a neighborhood of $K(f_1)$ onto a neighborhood of $K(f_2)$, which conjugates $f_1$ to $f_2$ and the complex dilatation of $h$ on $K(f_1)$ is zero. The following \textit{Straightening Theorem} is fundamental in the polynomial-like renormalization theory.

\begin{thm}[{\cite[p.\,296]{DH85b}}]\label{thm:straightening}
Let $f:U\to V$ be a polynomial-like mapping of degree $d\geqslant 2$. Then $f:U\to V$ is hybrid equivalent to a polynomial $P$ with the same degree $d$. Moreover, if $K(f)$ is connected, then $P$ is uniquely determined up to a conjugation by an affine map.
\end{thm}

In this paper we are only interested in \textit{quadratic-like mappings}, i.e., the polynomial-like mappings with degree $d=2$. By Theorem \ref{thm:straightening}, every quadratic-like mapping $f:U\to V$ is hybrid equivalent to a quadratic polynomial
\begin{equation}
Q_c(z)=z^2+c \text{\quad where\quad}c\in\C.
\end{equation}
If the unique critical orbit of $f$ is contained in $U$, then $Q_c$ is unique. We use $\beta_c$ to denote the \textit{$\beta$-fixed point} (i.e., the landing point of the external ray with angle zero) of $Q_c$ and $\beta_c'$ the other preimage of $\beta_c$.

If there exist $p\geqslant 1$ and two Jordan domains $U$, $V$ such that $f^{\circ p}:U\to V$ is a quadratic-like mapping with connected Julia set\footnote{If $f$ is a rational map, it is required that the domain $U$ does not contain the whole Julia set of $f$.}, then $f$ is called \textit{$p$-renormalizable} (or \textit{renormalizable} in short).
The sets $K(f)$, $f(K(f))$, $\cdots$, $f^{\circ (p-1)}(K(f))$ are called \textit{small filled Julia sets}, and $J(f)$, $f(J(f))$, $\cdots$, $f^{\circ (p-1)}(J(f))$ are called \textit{small Julia sets}.
The $\beta$-fixed point $\beta(f^{\circ p})$ of $f^{\circ p}:U\to V$ is defined as $h^{-1}(\beta_c)$ and the other preimage of $\beta(f^{\circ p})$ is defined as $\beta'(f^{\circ p}):=h^{-1}(\beta_c')$, where $c\in\C$ is uniquely determined such that $f^{\circ p}$ is hybrid conjugate to $Q_c$ by a quasiconformal mapping $h$.
For more backgrounds and results on the polynomial-like renormlization, we refer to \cite[Section 5]{McM94b}.

\subsection{Dynamics of McMullen maps}

For the family $f_\lambda$ in \eqref{equ:McM-3}, the critical points of $f_\lambda$ are $\{0,\infty\}\cup \Crit_\lambda$, where $\Crit_\lambda:=\{c_j=\lambda^{\frac{1}{2n}}e^{\frac{\pi\ii j}{n}}:0\leqslant j\leqslant 2n-1\}$ is the set of free critical points. Besides $\infty$, the rest critical values of $f_\lambda$ are $\pm 2\sqrt{\lambda}$.
By the dynamical symmetry (see \cite{DLU05}), $c_j$ lies in the attracting basin of $\infty$ if and only if $c_\ell$ does for $0\leqslant j,\ell\leqslant 2n-1$. Define the \textit{non-escape locus}
\begin{equation}
\mathcal{N}:=\big\{\lambda\in\C^*:f_\lambda^{\circ k}(c_0)\not\to\infty \text{ as }k\to\infty\big\}.
\end{equation}
Every component of $\C^*\setminus \mathcal{N}$ is a hyperbolic component and the unique unbounded hyperbolic component $\MH_0$ is called the \textit{Cantor locus} since $J(f_\lambda)$ is a Cantor set if and only if $\lambda\in\MH_0$. See the left picture in Figure \ref{Fig:Julia-gasket-McM}.

\begin{figure}[!htpb]
  \setlength{\unitlength}{1mm}
  \centering
  \includegraphics[width=0.45\textwidth]{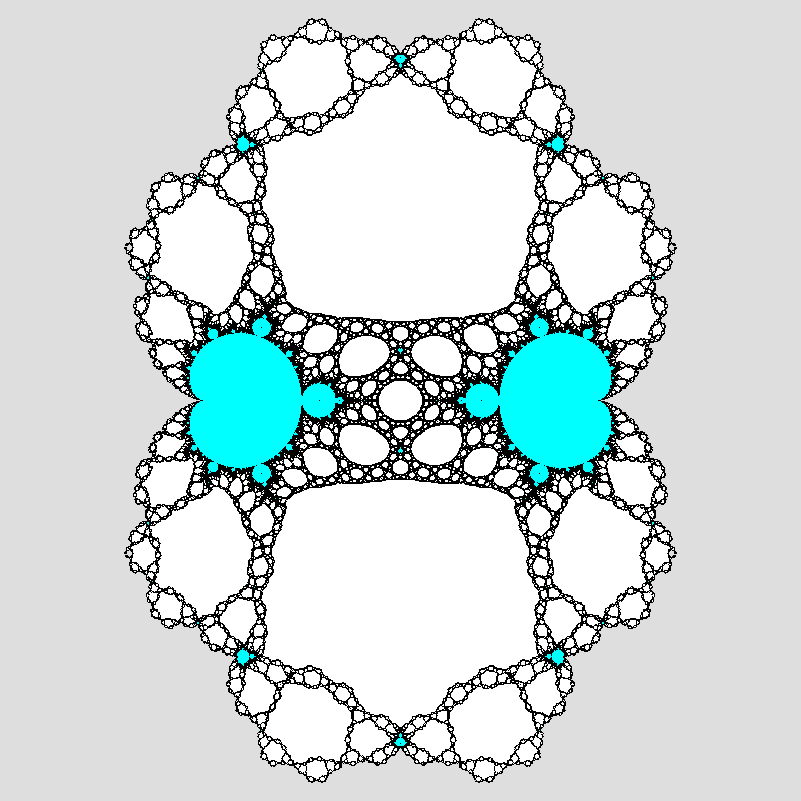}\quad
  \includegraphics[width=0.45\textwidth]{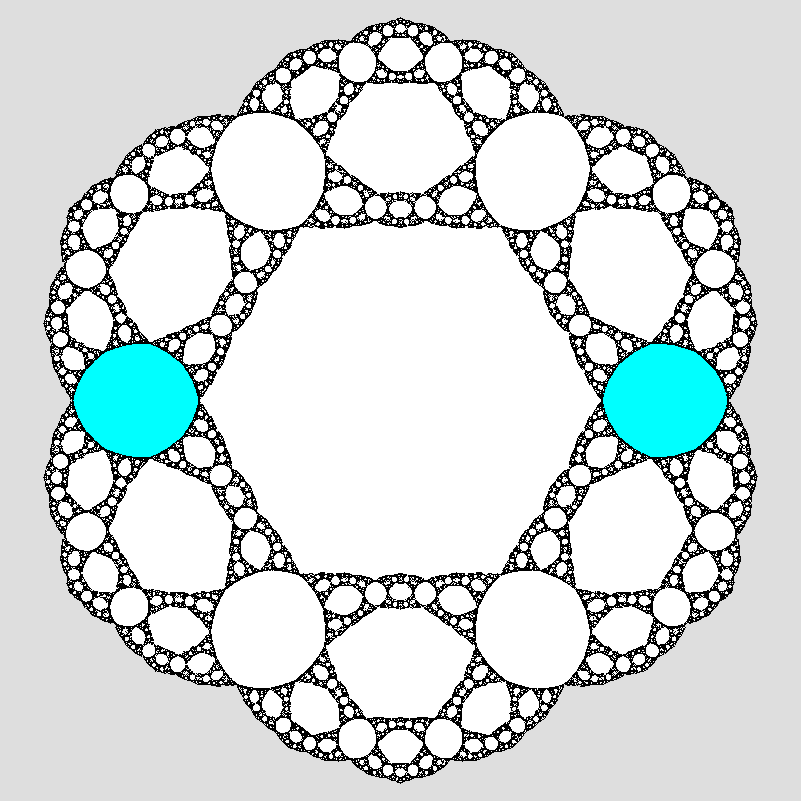}
  \put(-12,57){$B_\lambda$}
  \put(-34,31){$T_\lambda$}
  \put(-82,57){$\MH_0$}
  \caption{The parameter plane of $f_\lambda(z)=z^3+\lambda/z^3$ (the unbounded hyperbolic component $\MH_0$ is colored gray) and a corresponding generalized Sierpi\'{n}ski gasket Julia set, where $\lambda_0=1/8$ such that $f_{\lambda_0}$ is $1$-renormalizable and has two super-attracting fixed points $\pm 1/\sqrt{2}$ (whose immediate basins are colored cyan).}
  \label{Fig:Julia-gasket-McM}
\end{figure}

For the McMullen family $f_\lambda$, we consider the following renormalizations. If there exist a symbol $\epsilon\in\{\pm 1\}$, an integer $p\geqslant 1$, a critical point $c_j\in\Crit_\lambda$ and two Jordan domains $U$, $V$ containing $c_j$ such that $\epsilon f_\lambda^{\circ p}:U\to V$ is a quadratic-like mapping with connected Julia set, then $f_\lambda$ is called \textit{$p$-renormalizable} if $\epsilon=1$ and \textit{$p$-$*$-renormalizable} if $\epsilon=-1$ (or respectively, \textit{renormalizable} and \textit{$*$-renormalizable} in short). See \cite[Section 5]{QWY12} for more details.
Let $B_\lambda$ be the immediate super-attracting basin of $\infty$.
We need the following results from \cite[Theorems 1.1, 1.3 and Proposition 7.6]{QWY12} and \cite[Theorem 6.5]{QRWY15}.

\begin{thm}\label{thm:QRWY}
The following statements hold for the family $f_\lambda$ defined in \eqref{equ:McM-3}:
\begin{enumerate}
\item if $\lambda\not\in\MH_0$, then $\partial B_\lambda$ is a Jordan curve;
\item $\lambda\in\partial\MH_0$ if and only if $\partial B_\lambda$ contains either the critical set $\Crit_\lambda$ or a parabolic cycle of $f_\lambda$;
\item if $f_\lambda$ has no Siegel disk and $J(f_\lambda)$ is connected, then every Fatou component of $f_\lambda$ is a Jordan domain, and if further $f_\lambda$ is neither renormalizable nor $*$-renormalizable, then $J(f_\lambda)$ is locally connected.
\end{enumerate}
\end{thm}

As a consequence of Theorem \ref{thm:QRWY}, we obtain the following result.

\begin{cor}\label{cor:mcm-lc}
If $\lambda\in\partial\mathcal{H}_0$, then $J(f_\lambda)$ is connected, locally connected and every Fatou component of $f_\lambda$ is a Jordan domain.
\end{cor}

\begin{proof}
Let $\lambda\in\partial\MH_0$ be given.
Note that $\partial\MH_0$ is contained in the non-escape locus $\mathcal{N}$. By \cite{DR13}, $J(f_\lambda)$ is connected.
By Theorem \ref{thm:QRWY}(1) and (2), $f_\lambda$ has no Siegel disk since each point on the boundary of a Siegel disk is accumulated by the forward orbits of critical points (see \cite[Theorem 11.17, p.\,138]{Mil06}) but the forward orbit of the critical set $\Crit_\lambda$ is contained in $\partial B_\lambda$. Hence by Theorem \ref{thm:QRWY}(3), every Fatou component of $f_\lambda$ is a Jordan domain.

If $\partial B_\lambda$ contains a parabolic cycle of $f_\lambda$, then $f_\lambda$ is geometrically finite and $J(f_\lambda)$ is locally connected by \cite{TY96}.
Suppose $\partial B_\lambda$ contains the critical set $\Crit_\lambda$. Then $\lambda\not\in\R$ since otherwise $f_\lambda$ would have a fixed critical point $c_0$ and this is a contradiction. If $f_\lambda$ is critically finite, then $J(f_\lambda)$ is locally connected. Suppose $f_\lambda$ is critically infinite. By \cite[Proposition 4.1 and Lemma 7.1]{QWY12}, $f_\lambda$ is neither renormalizable nor $*$-renormalizable. According to Theorem \ref{thm:QRWY}(3), $J(f_\lambda)$ is locally connected.
\end{proof}

\subsection{Proof of Theorem \ref{thm:McM-family}}

Recall that $B_\lambda$ is the immediate super-attracting basin of $\infty$. Let $T_\lambda$ be the Fatou component of $f_\lambda$ containing $0$.

\begin{lem}[{\cite[Section 7]{Ste06}, \cite[Proposition 7.5]{QWY12} and see Figure \ref{Fig:Julia-gasket-McM}}]\label{lem:copy}
If $\partial B_\lambda$ contains a parabolic cycle, then
\begin{enumerate}
\item there exists an integer $p\geqslant 1$ such that $f_\lambda$ is either $p$-renormalizable or $p$-$*$-renormalizable and hybird equivalent to $Q_{1/4}(z)=z^2+1/4$ by a quasiconformal mapping $h$;
\item $K(\epsilon f_\lambda^{\circ p})\cap\partial B_\lambda=\{h^{-1}(\frac{1}{2})\}$ and $K(\epsilon f_\lambda^{\circ p})\cap\partial T_\lambda=\{h^{-1}(-\frac{1}{2})\}$, where $\epsilon\in\{\pm 1\}$ corresponds to renormalizable and $*$-renormalizable respectively.
\end{enumerate}
\end{lem}

If $f_\lambda$ contains a parabolic cycle, then every periodic Fatou component can only be super-attracting or parabolic, and $B_\lambda$ is the unique immediate super-attracting basin.

\begin{lem}\label{lem:no-intersect}
Suppose $\partial B_\lambda$ contains a parabolic cycle. Let $\MA$ and $\MP$ be the collection of components of the super-attracting basin and all parabolic basins of $f_\lambda$ respectively. Then
\begin{enumerate}
\item $\partial U\cap\partial V=\emptyset$ for any different components $U,V\in\MA$;
\item $\partial U\cap\partial V=\emptyset$ for any different components $U,V\in\MP$;
\item $\partial U$ and $\partial V$ share at most one point for any $U\in\MA$ and $V\in\MP$.
\end{enumerate}
\end{lem}

\begin{proof}
(1) Note that $f_\lambda(B_\lambda)=B_\lambda$, $f_\lambda^{-1}(B_\lambda)=B_\lambda\cup T_\lambda$ and $\MA=\bigcup_{k\geqslant 0}f_\lambda^{-k}(B_\lambda)$. Since $\partial B_\lambda$, $\partial T_\lambda$ are Jordan curves and $\partial B_\lambda$ does not contain any critical points, it follows that $\partial B_\lambda\cap\partial T_\lambda=\emptyset$. Inductively, for any different components $U,V\in\MA$, we have $\partial U\cap\partial V=\emptyset$.

\medskip
(2) By Lemma \ref{lem:copy}(1), without loss of generality, we assume that there exists $p\geqslant 1$ such that $f_\lambda$ is $p$-renormalizable.
According to \cite[Propositions 4.1 and 7.5]{QWY12}, the filled Julia set $K(f_\lambda^{\circ p})$ is the intersection of a sequence of Yoccoz puzzles, whose boundaries consist of equipotential curves and ``cut rays" in $\MA\cup J(f_\lambda)$. Let $\MP_0$ be the interior of $K(f_\lambda^{\circ p})$. Then $\MP_0$ is a Jordan domain, and for any component $U$ of $\MP\setminus\MP_0$, we have $\partial U\cap\partial \MP_0=\emptyset$.

By the dynamical symmetry (see \cite{DLU05}), without loss of generality, we assume that $\MP_0$ contains the critical point $c_0=\lambda^{\frac{1}{2n}}$. Moreover, there are other $2n-1$ Fatou components $\MP_j=e^{\frac{\pi\ii j}{n}}\MP_0$ containing the critical points $c_j=\lambda^{\frac{1}{2n}}e^{\frac{\pi\ii j}{n}}$ respectively, where $1\leqslant j\leqslant 2n-1$.
Note that $f_\lambda$ has exactly two critical values $\pm 2\sqrt{\lambda}$ in $\C\setminus\{0\}$. It implies that each $\MP_j$ is strictly preperiodic for $j\neq 0, n$ and $f_\lambda$ is $p$-renormalizable exactly at the critical points $c_0$, $c_n$ for odd $n\geqslant 3$ or $p$-renormalizable exactly at the critical point $c_0$ for even $n\geqslant 4$.
We only consider the odd $n\geqslant 3$ since the case for the even $n\geqslant 4$ is completely similar. Therefore, we have
\begin{equation}\label{equ:P}
\MP=\bigcup_{k\geqslant 0}f_\lambda^{-k}(\MP_0\cup\MP_n).
\end{equation}
By a similar argument as above, for any component $U$ of $\MP\setminus\MP_n$, we have $\partial U\cap\partial \MP_n=\emptyset$.
Note that each component of $\MP$ is a Jordan domain whose boundary does not contain critical points.
Inductively, we have $\partial U\cap\partial V=\emptyset$ for any different components $U,V\in\MP$.

\medskip
(3) Without loss of generality, we assume that $n\geqslant 3$ is odd. Then $f_\lambda$ has two $p$-cycles of parabolic Fatou components
$\{\MP_0$, $f_\lambda(\MP_0)$, $\cdots$, $f_\lambda^{\circ (p-1)}(\MP_0)\}$ and $\{\MP_n$, $f_\lambda(\MP_n)$, $\cdots$, $f_\lambda^{\circ (p-1)}(\MP_n)\}$. If $p\geqslant 2$, then \begin{equation}\label{equ:f-conf}
f_\lambda^{\circ (p-i)}:f_\lambda^{\circ i}(\MP_0)\to\MP_0 \text{\quad and\quad} f_\lambda^{\circ (p-i)}:f_\lambda^{\circ i}(\MP_n)\to\MP_n
\end{equation}
are conformal for all $1\leqslant i\leqslant p-1$.
By Lemma \ref{lem:copy}(2), $\sharp(\partial\MP_j\cap\partial B_\lambda)=1$ and $\sharp(\partial\MP_j\cap\partial T_\lambda)=1$ for all $0\leqslant j\leqslant 2n-1$.
Hence by \eqref{equ:f-conf}, for $0\leqslant j\leqslant 2n-1$, we have $\sharp(f_\lambda^{\circ i}(\partial\MP_j)\cap\partial B_\lambda)=1$ for all $0\leqslant i\leqslant p-1$ and $f_\lambda^{\circ i}(\partial\MP_j)\cap\partial T_\lambda=\emptyset$ for all $1\leqslant i\leqslant p-1$ since $f_\lambda^{-1}(\partial B_\lambda)=\partial B_\lambda\cup\partial T_\lambda$.

Let $U\in\MA\setminus B_\lambda$ and $V\in\MP$ such that $\partial U\cap\partial V\neq\emptyset$ and $k_1$, $k_2\geqslant 0$ be the smallest integers such that $f_\lambda^{\circ k_1}(U)=T_\lambda$ and $f_\lambda^{\circ k_2}(V)=\MP_j$ for some $0\leqslant j\leqslant 2n-1$. By \eqref{equ:f-conf}, the maps $f_\lambda^{\circ k_1}:U\to T_\lambda$ and $f_\lambda^{\circ k_2}:V\to\MP_j$ are both conformal.
If $k_2>k_1$, then $f_\lambda^{\circ k_2}(U)=B_\lambda$. Since $\sharp(\partial\MP_j\cap\partial B_\lambda)=1$, we have $\sharp(\partial U\cap\partial V)=1$.
If $k_2=k_1$, since $\sharp(\partial\MP_j\cap\partial T_\lambda)= 1$, we have $\sharp(\partial U\cap\partial V)=1$.
If $k_2<k_1$, then $f_\lambda^{\circ k_1}(V)=f_\lambda^{\circ (k_1-k_2)}(\MP_j)$. Since $\sharp(f_\lambda^{\circ (k_1-k_2)}(\partial\MP_j)\cap\partial T_\lambda)\leqslant 1$, we have $\sharp(\partial U\cap\partial V)=1$. Similarly, one can obtain $\sharp(\partial B_\lambda\cap\partial V)\leqslant 1$ for all $V\in\MP$ and the proof is complete.
\end{proof}

The following result is included in \cite{DRS07} for the special case when $f_\lambda$ is critically finite with $\lambda\in\partial\MH_0$ (see Theorem \ref{thm:QRWY}(2)). Although the idea is similar, we include a proof here for the completeness.

\begin{lem}\label{lem:crit-on-B}
Suppose $\partial B_\lambda$ contains the critical set $\Crit_\lambda$. Let $\MA$ be the collection of components of the super-attracting basin of $f_\lambda$. Then $\partial U\cap\partial V$ share at most $2n$ points  for any different components $U$ and $V$ of $\MA$.
\end{lem}

\begin{proof}
By Theorem \ref{thm:QRWY}(2) and Corollary \ref{cor:mcm-lc}, $\EC\setminus\overline{B}_\lambda$ is a Jordan domain and contains no critical values. Hence $f_\lambda^{-1}(\EC\setminus\overline{B}_\lambda)$ consists of exactly $2n$ components $W_j$ which are Jordan domains and $f_\lambda: W_j\to\EC\setminus\overline{B}_\lambda$ is conformal (see Lemma \ref{lem:Jordan}(3)), where $W_j=e^{\frac{\pi\ii j}{n}}W_0$ for $0\leqslant j\leqslant 2n-1$. Since $f_\lambda^{-1}(B_\lambda)=B_\lambda\cup T_\lambda$, we have $\partial B_\lambda\cap\partial T_\lambda=\Crit_\lambda$ and $\sharp(\partial B_\lambda\cap\partial T_\lambda)=2n$.

Let $U_j\subset W_j$ such that $f_\lambda(U_j)=T_\lambda$, where $0\leqslant j\leqslant 2n-1$. Then $f_\lambda^{-1}(T_\lambda)=\bigcup_{0\leqslant j\leqslant 2n-1} U_j$. Moreover, for each $0\leqslant j\leqslant 2n-1$, there are exactly $2n$ points $\{c_{j,\ell}:0\leqslant \ell\leqslant 2n-1\}=\partial W_j\cap\partial U_j$ such that $f_\lambda(c_{j,\ell})=c_\ell\in\partial B_\lambda\cap\partial T_\lambda$. Note that
\begin{equation}
\{c_{j,\ell}:0\leqslant \ell\leqslant 2n-1\}\cap\Crit_\lambda=\emptyset
\end{equation}
for any $0\leqslant j\leqslant 2n-1$ since otherwise some critical point $c_\ell$ would be periodic, which contradicts to the fact that $\Crit_\lambda\subset\partial B_\lambda$. This implies that $W_j\setminus\overline{U}_j$ consists of $2n$ components $\{W_{j,\ell}:0\leqslant \ell\leqslant 2n-1\}$ which are Jordan domains such that $f_\lambda:W_{j,\ell}\to W_\ell$ is conformal. Thus we have $\sharp(\partial U_j\cap\partial B_\lambda)\leqslant 2n$, $\sharp(\partial U_j\cap\partial T_\lambda)\leqslant 2n$ and $\sharp(\partial U_i\cap\partial U_j)=\emptyset$ for $0\leqslant i\neq j\leqslant 2n-1$.

Since $f_\lambda: W_j\to\EC\setminus\overline{B}_\lambda$ is conformal, inductively, one can obtain $W_{j_1,\cdots,j_k}$, $U_{j_1,\cdots,j_k}$, $c_{j_1,\cdots,j_k}$ for all $k\geqslant 2$, where $0\leqslant j_k\leqslant 2n-1$, such that
\begin{itemize}
\item $\overline{W}_{j_1,\cdots,j_k}=\overline{U}_{j_1,\cdots,j_k}\cup \{\overline{W}_{j_1,\cdots,j_k,j_{k+1}}:0\leqslant j_{k+1}\leqslant 2n-1\}$;
\item $f_\lambda(U_{j_1,\cdots,j_k})=U_{j_2,\cdots,j_k}$ and $f_\lambda(W_{j_1,\cdots,j_{k+1}})=U_{j_2,\cdots,j_{k+1}}$;
\item $\partial W_{j_1,\cdots,j_k}\cap\partial U_{j_1,\cdots,j_k}=\{c_{j_1,\cdots,j_k,j_{k+1}}:0\leqslant j_{k+1}\leqslant 2n-1\}$;
\item $c_{j_1,\cdots,j_k}=c_{j_1',\cdots,j_\ell'}$ if and only if $k=\ell$ and $(j_1,\cdots,j_k)=(j_1',\cdots,j_\ell')$;
\item $\sharp(\partial U_{j_1,\cdots,j_k}\cap\partial B_\lambda)\leqslant 2n$, $\sharp(\partial U_{j_1,\cdots,j_k}\cap\partial T_\lambda)\leqslant 2n$ and $\sharp(\partial U_{j_1,\cdots,j_k}\cap\partial U_{j_1',\cdots,j_\ell'})\leqslant 2n$, where $0\leqslant j_1',\cdots,j_\ell'\leqslant 2n-1$ and $1\leqslant \ell\leqslant k$.
\end{itemize}
Note that
\begin{equation}\label{equ:A}
\MA=\bigcup_{k\geqslant 0}f_\lambda^{-k}(B_\lambda)=B_\lambda\cup T_\lambda\cup\bigcup_{k\geqslant 1}~\bigcup_{0\leqslant j_1,\cdots,j_k\leqslant 2n-1}U_{j_1,\cdots,j_k}.
\end{equation}
Thus for any two different components $U$ and $V$ of $\MA$, we have $\sharp(\partial U\cap\partial V)\leqslant 2n$.
\end{proof}

\begin{proof}[Proof of Theorem \ref{thm:McM-family}]
Based on Corollary \ref{cor:mcm-lc}, Lemmas \ref{lem:no-intersect} and \ref{lem:crit-on-B}, it suffices to verify the last two properties in the definition of generalized Sierpi\'{n}ski gaskets. By Theorem \ref{thm:QRWY}(2), we divide the argument into two cases.

\medskip
\textbf{Case 1}: \textit{$\partial B_\lambda$ contains a parabolic cycle of $f_\lambda$}.
Let $\MA$ and $\MP$ be the collection of components of the super-attracting basin and all parabolic basins of $f_\lambda$ respectively. Assume there are three different Fatou components of $f_\lambda$ having a common boundary point. Then at least two of them belong to either $\MA$ or $\MP$. By Lemma \ref{lem:no-intersect}(1)(2), such two Fatou components have no common boundary point. Hence no three Fatou components of $f_\lambda$ have a common boundary point.

By Lemma \ref{lem:copy}, without loss of generality, we assume that $f_\lambda$ is $p$-renormalizable at the critical point $c_0=\lambda^{\frac{1}{2n}}$ and that $n\geqslant 3$ is odd.
By \eqref{equ:P}, the Fatou set of $f_\lambda$ is
\begin{equation}
F(f_\lambda)=\bigcup_{k\geqslant 0}\MU_k, \quad \text{where } \MU_k=f_\lambda^{-k}\Big(B_\lambda\cup\bigcup_{j=0}^{p-1}f_\lambda^{\circ j}(\MP_0\cup\MP_n)\Big).
\end{equation}
By Lemma \ref{lem:copy}, $\overline{\MU}_0$ is connected and contains all critical values of $f_\lambda$. Since each component of $\EC\setminus\overline{\MU}_0$ is simply connected and $\MU_1=f_\lambda^{-1}(\MU_0)$, it follows that each component of $\EC\setminus\overline{\MU}_1$ is simply connected by the Riemann-Hurwitz formula and $\overline{\MU}_1$ is connected. Inductively, $\overline{\MU}_k$ is connected for all $k\geqslant 0$.
Thus the contact graph corresponding to $J(f_\lambda)$ is connected.

\medskip
\textbf{Case 2}: \textit{$\partial B_\lambda$ contains the critical set $\Crit_\lambda$}.
Assume that there are three different Fatou components $V_0$, $V_1$ and $V_2$ of $f_\lambda$ having a common boundary point $z_0$. By \eqref{equ:A}, except $B_\lambda$ and $T_\lambda$, every Fatou component of $f_\lambda$ can be written as $U_{j_1,\cdots,j_k}$, whose closure does not contain any critical point.
Since $F(f_\lambda)=\bigcup_{k\geqslant 0}f_\lambda^{-k}(B_\lambda)$ and $\partial B_\lambda\cap\partial T_\lambda=\Crit_\lambda$, iterating several times if necessary, we can assume that $V_0=B_\lambda$, $V_1=T_\lambda$ and $z_0\in\Crit_\lambda$. This is a contradiction since $\partial V_2\cap\Crit_\lambda=\emptyset$.
Hence no three Fatou components of $f_\lambda$ have a common boundary point.

Note that $\overline{B}_\lambda$ is connected and contains all critical values of $f_\lambda$. Since $\EC\setminus\overline{B}_\lambda$ is simply connected, it follows that each component of $\EC\setminus f_\lambda^{-1}(\overline{B}_\lambda)=\EC\setminus(\overline{B}_\lambda\cup \overline{T}_\lambda)$ is simply connected by the Riemann-Hurwitz formula and therefore $f_\lambda^{-1}(\overline{B}_\lambda)$ is connected. Inductively, $f_\lambda^{-k}(\overline{B}_\lambda)$ is connected for all $k\geqslant 0$.
Thus the contact graph corresponding to $J(f_\lambda)$ is connected.
\end{proof}

Besides $\lambda\in\partial\MH_0$, we can find some gasket Julia sets for $\lambda\not\in\MH_0$ as following. According to \cite[Section 7]{Ste06} and \cite[Section 5]{QWY12}, there exists a homeomorphic copy $\mathcal{M}$ of the Mandelbrot set $M$ in the non-escape locus $\mathcal{N}$ such that for each $c\in M$, there exist $\lambda=\lambda(c)\in\MM$ and two Jordan domains $U_\lambda$, $V_\lambda$ such that $f_\lambda:U_\lambda\to V_\lambda$ is a quadratic-like mapping and hybird equivalent to $Q_c$, i.e., $f_\lambda$ is $1$-renormalizable. Moreover, the filled Julia set of $f_\lambda$ in $U_\lambda$ intersects $\partial B_\lambda$ and $\partial T_\lambda$ at exactly one point respectively.

The case when $Q_c(z)=z^2+1/4$ is considered in Theorem \ref{thm:McM-family}. In fact, if $Q_c$ has an attracting fixed point (including super-attracting), then the corresponding McMullen map $f_\lambda$ with $\lambda=\lambda(c)$ is hyperbolic and the Julia set $J(f_\lambda)$ is still a gasket Julia set. See the right picture in Figure \ref{Fig:Julia-gasket-McM}.

\section{Area of gasket Julia sets}\label{sec:area-dim}

As far as we know, except the examples in Theorem \ref{thm:McM-family}, all the known gasket Julia sets are generated by geometrically finite rational maps and they have Hausdorff dimension strictly less than two. In this section we study the area of generalized Sierpi\'{n}ski gasket Julia sets.

\subsection{A criterion to obtain Jordan domains}

In order to show that a Julia set is a generalized Sierpi\'{n}ski gasket, it is necessary to prove that every Fatou component is a Jordan domain. In general, we start from the periodic Fatou components and then consider their preimages. The following criterion is useful to obtain Jordan domains.

\begin{lem}\label{lem:Jordan}
Let $f:U'\to V'$ be a proper holomorphic map. Suppose $V$ is a Jordan domain which is compactly contained in $V'$ and $U$ is a component of $f^{-1}(V)$ in $U'$. Assume that one of the following conditions is satisfied:
\begin{enumerate}
\item $\partial U$ does not contain any critical points and $U$ is simply connected; or
\item $\overline{V}$ contains at most one critical value; or
\item $V$ contains no critical values.
\end{enumerate}
Then $U$ is a Jordan domain.
\end{lem}

\begin{proof}
(1) If $\partial U$ does not contain critical points, then $\partial V$ does not contain any critical values coming from $f|_{\partial U}$. Hence $\partial U$ is locally a Jordan arc since $\partial V$ is a Jordan curve. Note that any nonempty compact connected $1$-dimensional manifold is homeomorphic to the unit circle (see \cite[Theorem 5.27, p.\,145]{Lee11}). It follows that each component of $\partial U$ is a Jordan curve since $f$ is proper. If $U$ is simply connected, then $\partial U$ has exactly one connected component and hence $U$ is a Jordan domain.

Parts (2) and (3) have been proved in \cite[Proposition 2.8]{Pil96} essentially. For the completeness we include a proof here with similar ideas but the argument is slightly different.
For (2), if $V$ contains exactly one critical value, then $U$ is simply connected by the Riemann-Hurwitz formula and $\partial U$ does not contain any critical points. By (1) we conclude that $U$ is a Jordan domain.

To finish the proof, it is sufficient to assume that $V$ contains no critical values. Note that $f:U\to V$ is conformal, we only need to prove that $f:\partial U\to\partial V$ is injective. Otherwise, there exist two simple arcs $\gamma_1,\gamma_2$ in $U$ whose interiors are disjoint, connecting $x_0\in U$ with different $x_1,x_2\in\partial U$ respectively, such that $f(x_1)=f(x_2)=y\in\partial V$.
Then $f(\gamma_1 \cup \gamma_2)$ is a Jordan curve in $\overline{V}$ which intersects $\partial V$ at exactly one point $y$. Hence, there exists a component $W$ of $U\setminus(\gamma_1 \cup \gamma_2)$ such that $f(\partial W \cap \partial U)=y$. Since $x_1 \neq x_2$, the set $\partial W \cap \partial U$ has accumulation points. This is impossible since $f$ is holomorphic in a neighborhood of $\overline{U}$.
\end{proof}

If $V$ and $\partial V$ each contains a critical value and $U$ is simply connected, then $U$ is not necessarily a Jordan domain. For example, consider the rational map $f(z)=z+1/z$ and $V=\{z\in\C:|z-2|<4\}$. Then $V$ and $\partial V$ contain the critical values $2$ and $-2$ of $f$ respectively. Let $U$ be a connected component of $f^{-1}(V)$. By the Riemann-Hurwitz formula, $f:U\to V$ has degree two and hence $f^{-1}(V)$ has exactly one component, which is $U$. But $U$ is not a Jordan domain since otherwise $f:\EC\setminus\overline{U}\to\EC\setminus\overline{V}$ would have degree two, which is a contradiction since $\EC\setminus\overline{V}$ has no critical values.

\subsection{Area of quadratic Julia sets}

In this subsection, we recall some area results on the Julia sets of quadratic polynomials.
A rational map $f$ is \textit{infinitely renormalizable} if it is $p$-renormalizable for infinitely many $p$'s (see Section \ref{subsec:poly-like} for the definition of $p$-renormalization).
Note that irrationally indifferent periodic points correspond to two types: Cremer points (locally non-linearizable) and Siegel disks (locally linearizable).
In the past two decades, the quadratic polynomial Julia sets with positive area were found in all these three cases.

\begin{thm}[{\cite{BC12}, \cite{AL22}, \cite{DL23a}}]\label{thm:positive-area}
There exist quadratic polynomials that have Julia sets with positive area. Such quadratic polynomials either contain a Cremer fixed point, a Siegel disk or are infinitely renormalizable.
\end{thm}

Buff and Ch\'{e}ritat proved the existence of quadratic Julia sets with positive area in three cases: Cremer, Siegel and infinitely renormalizable. Avila and Lyubich, Dudko and Lyubich found other two different types of infinitely renormalizable quadratic polynomials having Julia sets with positive area.

\subsection{A quadratic family with Siegel disks}

The Julia set of any polynomial cannot be a gasket. To construct gasket Julia sets, we first consider quadratic rational maps. In the definition of generalized Sierpi\'{n}ski gasket, the connectivity of the contact graph corresponding to the Julia set should be given priority consideration. This inspires us to consider the family of quadratic rational maps having a fixed bounded type Siegel disk since the immediate preimage of this Siegel disk is attached at its boundary. Specifically, let $\theta$ be a bounded type irrational number. Define
\begin{equation}
q_\lambda(z)=\frac{z^2-\lambda}{e^{2\pi\ii\theta}(z-\frac{1+\lambda}{2})},
\end{equation}
where $\lambda\in\C\setminus\{1\}$. A direct calculation shows that $q_\lambda$ has two critical points $1$ and $\lambda$. By \cite{Zha11}, $q_\lambda$ has a Siegel disk $\Delta_\lambda$ centered at $\infty$ with rotation number $\theta$ whose boundary is a Jordan curve (actually a quasi-circle) containing the critical point $1$, or $\lambda$, or both.

\begin{lem}\label{lem:gasket-pre}
Suppose $q_\lambda$ satisfies one of the following:
\begin{enumerate}
\item it is renormalizable and the small filled Julia set has no interior; or
\item the forward orbit of a critical point is finite in the Julia set; or
\item the forward orbit of a critical point intersects the closure of $\Delta_\lambda$.
\end{enumerate}
Then $J(q_\lambda)$ is connected, locally connected and every Fatou component of $q_\lambda$ is a Jordan domain.
\end{lem}

\begin{proof}
If $q_\lambda$ is in one of the three cases above, then except the Siegel disk $\Delta_\lambda$, the map $q_\lambda$ has no other periodic Fatou components \cite{Shi87}. Hence the Fatou set of $q_\lambda$ is $F(q_\lambda)=\bigcup_{k\geqslant 0}q_\lambda^{-k}(\Delta_\lambda)$. Let $U$ and $V$ be two different components of $F(q_\lambda)$ such that $V$ is a Jordan domain and $q_\lambda(U)=V$. We claim that $U$ is also a Jordan domain. Indeed, if $V$ contains no critical value or $\overline{V}$ contains at most one critical value, then $U$ is also a Jordan domain by Lemma \ref{lem:Jordan}(2)(3). Assume that $U$ and $\partial U$ each contain a critical point. Then the critical point on $\partial U$ also lies on $\partial\Delta_\lambda$ and hence $q_\lambda(U)=V=\Delta_\lambda$. However, $V\neq\Delta_\lambda$ since $\deg(q_\lambda)=2$. This is a contradiction and the claim is proved. By induction, all components of $F(q_\lambda)$ are Jordan domains. This implies that $J(q_\lambda)=\EC\setminus F(q_\lambda)$ is connected.

To obtain the local connectivity of $J(q_\lambda)$, according to the criterion in \cite[Theorem 4.4, pp.\,112-113]{Why42}, it suffices to prove that the spherical diameter of the Fatou components of $q_\lambda$ tends to zero. The postcritical set of $q_\lambda$ is $\MP(q_\lambda):=\overline{\bigcup_{k\geqslant 1} q_\lambda^{\circ k}(\{1,\lambda\})}$. Without loss of generality, suppose $\partial\Delta_\lambda$ contains the critical point $\lambda$. Then for Case (2), we have
\begin{equation}\label{equ:disjoint}
\overline{\bigcup_{k\geqslant 1} q_\lambda^{\circ k}(1)}\cap \bigcup_{k\geqslant 0} q_\lambda^{-k}(\overline{\Delta}_\lambda)=\emptyset.
\end{equation}
We claim that \eqref{equ:disjoint} is also true for Case (1). Indeed, suppose there exist $p\geqslant 1$ and two Jordan domains $U_0$, $V_0$ such that $q_\lambda^{\circ p}:U_0\to V_0$ is a quadratic-like mapping whose small filled Julia set $K_0$ contains the critical point $1$. It suffices to prove that $\partial\Delta_\lambda\cap K_0=\emptyset$. Otherwise, since the forward orbit of any point in $\partial\Delta_\lambda$ is dense in $\partial\Delta_\lambda$, we would have $\partial\Delta_\lambda\subset K_0$ and the other critical point $\lambda\in\partial\Delta_\lambda$ is also contained in $K_0$. This is a contradiction and the claim is proved.

If $q_\lambda$ is in one of the three cases in the lemma, then the spherical distance between $\MP(q_\lambda)\setminus\partial\Delta_\lambda$ and $\partial\Delta_\lambda$ is positive (including the case that $\MP(q_\lambda)\setminus\partial\Delta_\lambda=\emptyset$). There exists $k\geqslant 1$ such that for any component $V_0$ of $q_\lambda^{-k}(\Delta_\lambda)\setminus q_\lambda^{-(k-1)}(\Delta_\lambda)$, we have $\overline{V}_0\cap\MP(q_\lambda)\subset\partial\Delta_\lambda$. By \cite[Main Lemma]{WYZZ25} and the classical shrinking lemma (see \cite{TY96} or \cite[p.\,86]{LM97}), the spherical diameter of the preimages of $V_0$ under $q_\lambda$ tends to zero uniformly. Since the number of such $V_0$'s  is finite, we conclude that $J(q_\lambda)$ is locally connected.
\end{proof}

In fact, since $q_\lambda$ has a Siegel disk, the Julia set of $q_\lambda$ cannot be a Cantor set and hence is always connected for all $\lambda\in\C\setminus\{1\}$ (see \cite{Yin92}, \cite{Mil93}).

\begin{lem}\label{lem:gasket-Sie}
If the two critical points of $q_\lambda$ have disjoint forward orbits and $q_\lambda$ satisfies one of the cases in Lemma \ref{lem:gasket-pre}, then $J(q_\lambda)$ is a gasket Julia set.
\end{lem}

\begin{proof}
Based on Lemma \ref{lem:gasket-pre}, it suffices to verify the last three properties in the definition of generalized Sierpi\'{n}ski gaskets.

Note that $F(q_\lambda)=\bigcup_{k\geqslant 0}q_\lambda^{-k}(\Delta_\lambda)$.
Since $\partial\Delta_\lambda$ contains at most two critical points and $q_\lambda:\Delta_\lambda\to\Delta_\lambda$ is conformal, it follows that $\sharp(\partial U\cap\partial\Delta_\lambda)\leqslant 2$ for any component $U$ of $F(q_\lambda)$. Let $U_1$ and $U_2$ be two different components of $F(q_\lambda)\setminus\Delta_\lambda$ such that $\partial U_1\cap\partial U_2\neq\emptyset$, and $k_1$, $k_2$ be the smallest integers such that $q_\lambda^{\circ k_1}(U_1)=\Delta_\lambda=q_\lambda^{\circ k_2}(U_2)$, where $k_2\geqslant k_1\geqslant 1$. Then at least one of the maps $q_\lambda^{\circ k_1}:U_1\to\Delta_\lambda$ and $q_\lambda^{\circ k_1}:U_2\to q_\lambda^{\circ k_1}(U_2)$ is conformal.
If $k_2>k_1$, then $\sharp\big(\partial \Delta_\lambda\cap q_\lambda^{\circ k_1}(\partial U_2)\big)\leqslant 2$.
If $k_2=k_1$, then there exists a smallest $k_0\leqslant k_1$ such that $q_\lambda^{\circ k_0}(U_1)=q_\lambda^{\circ k_0}(U_2)$.
In both cases we have $\sharp(\partial U_1\cap\partial U_2)\leqslant 2$. Hence $\sharp(\partial U\cap\partial V)\leqslant 2$ for any different components $U$, $V$ of $F(q_\lambda)$.
In particular, $\sharp(\partial U\cap\partial V)\leqslant 1$ if $q_\lambda$ satisfies one of the first two cases in Lemma \ref{lem:gasket-pre}.

Assume that there are three Fatou components $U_0$, $U_1$ and $U_2$ of $q_\lambda$ having a common boundary point $z_0$. Since these Fatou components are eventually iterated to $\Delta_\lambda$, the forward orbit of $z_0$ must contain two critical points, which contradicts to the assumption that the two critical points of $q_\lambda$ have disjoint forward orbits.
Thus no three Fatou components of $q_\lambda$ have a common boundary point.

Let $\Delta_\lambda'\neq\Delta_\lambda$ be the unique component of $q_\lambda^{-1}(\Delta_\lambda)$ attaching at a critical point on $\partial\Delta_\lambda$. For any component $U_0$ of $F(q_\lambda)\setminus q_\lambda^{-1}(\Delta_\lambda)$, there exists an integer $k\geqslant 1$ such that $q_\lambda^{\circ k}(U_0)=\Delta_\lambda'$. Then there exists a Fatou component $U_1$ such that
\begin{equation}
\partial U_0\cap\partial U_1\neq\emptyset \text{\quad and\quad} q_\lambda^{\circ (k-1)}(U_1)=\Delta_\lambda.
\end{equation}
Inductively, there exists a sequence of Fatou components $U_0$, $U_1$, $\cdots$, $U_\ell=\Delta_\lambda$, where $\ell\geqslant 1$ such that
$\partial U_i\cap \partial U_{i+1}\neq\emptyset \text{ for all } 0\leq i\leqslant \ell-1$.
Thus the contact graph corresponding to $J(q_\lambda)$ is connected. The proof is complete.
\end{proof}

If $\lambda=4/e^{2\pi\ii\theta}-1$, then $q_\lambda$ has a finite critical orbit $1\mapsto \frac{1+\lambda}{2}\mapsto\infty\in\Delta_\lambda$. Hence according to Lemmas \ref{lem:gasket-pre}(3) and \ref{lem:gasket-Sie}, the Julia set of $q_\lambda$ is a generalized Sierpi\'{n}ski gasket. In fact, in this case $q_\lambda$ is conjugate to $z\mapsto e^{2\pi\ii\theta}z/(1-z)^2$. See the right picture of Figure \ref{Fig:Julia-gasket-Siegel} (where $\theta=(\sqrt{5}-1)/2$ is of bounded type). Such Julia sets have Hausdorff dimension strictly less than two (see \cite[Theorem 3.14]{Lim24b}).

\vskip0.2cm
To obtain the gasket Julia sets with positive area, we use polynomial-like renormalization theory.
The following result concerning the universality of the Mandelbrot set was proved by McMullen.

\begin{lem}[{\cite[Theorem 1.3]{McM00b}}]\label{lem:McM-univ}
There exist $p\geqslant 1$ and a homeomorphic copy $\mathcal{M}$ of the Mandelbrot set $M$ in parameter space $\C\setminus\{1\}$ of $q_\lambda$ such that for each $c\in M$, there exist $\lambda=\lambda(c)\in\MM$ and two Jordan domains $U_\lambda$, $V_\lambda$ such that $q_\lambda^{\circ p}:U_\lambda\to V_\lambda$ is a quadratic-like mapping and hybird equivalent to $Q_c(z)=z^2+c$.
\end{lem}

\begin{proof}[Proof of Theorem \ref{thm:area-dim}]
By Theorem \ref{thm:positive-area}, let $c\in M$ such that $Q_c(z)=z^2+c$ has a Cremer point or is infinitely renormalizable and that $J(Q_c)$ has positive area. By Lemma \ref{lem:McM-univ}, there exist $p\geqslant 1$, $\lambda=\lambda(c)\in\C\setminus\{1\}$ and two Jordan domains $U_\lambda$, $V_\lambda$ such that $q_\lambda^{\circ p}:U_\lambda\to V_\lambda$ is a quadratic-like mapping and hybird equivalent to $Q_c(z)=z^2+c$. Since quasiconformal mappings are absolutely continuous with respect to the $2$-dimensional Lebesgue measure, this implies that $J(q_\lambda)$ has positive area. By Lemma \ref{lem:gasket-Sie}, $J(q_\lambda)$ is a Sierpi\'{n}ski gasket.
\end{proof}

We would like to mention that one can also use the family $b_\mu(z)=\mu/(z^2+2z)$, where $\mu\in\C\setminus\{0\}$, to generate the gasket Julia sets (see \cite[Section 5]{LZ25}), and even those with positive area. Each $b_\mu$ has a super-attracting $2$-cycle $\{0,\infty\}$ and the Julia set $J(b_\mu)$ contains a \textit{basilica-like} structure (see \cite{AY09} and the references therein). One can use the similar idea as Lemmas \ref{lem:gasket-pre} and \ref{lem:gasket-Sie} to ``arrange" a copy of a quadratic Julia set with positive area in the gasket Julia set of $b_\mu$.


\section{A criterion to generate gasket Julia sets}\label{sec:criterion}

In this section, we first give a criterion to generate some rational maps having gasket Julia sets and then use this criterion to study a specific rational family.

\subsection{A criterion}

Comparing the conditions in Theorem \ref{thm:criterion} and those in \cite[Theorem A]{Yan18}, one may observe that they are very similar but also have some differences: the latter rational maps are hyperbolic while the former rational maps are geometrically finite but not hyperbolic. Similar to Corollary \ref{cor:mcm-lc} and Lemma \ref{lem:gasket-pre}, we first establish the following result.

\begin{lem}\label{lem:criterion-lc}
Let $f$ be the rational map in Theorem \ref{thm:criterion}. Then $J(f)$ is connected, locally connected and every Fatou component of $f$ is a Jordan domain.
\end{lem}

\begin{proof}
Suppose that $f$ satisfies all the assumptions in Theorem \ref{thm:criterion}. Then $f$ is geometrically finite and $U$ is the only periodic Fatou component of $f$, which is either attracting, super-attracing or parabolic. The Julia set of $f$ is connected and locally connected according to \cite{TY96}.
By a completely similar argument (which is based on the first three conditions) to the proof of \cite[Theorem A, p.\,2134]{Yan18}, the Fatou component $U$ is a Jordan domain.

Since $f(U)=f(V)=U$, it follows that every point in $\partial U\cap\partial V$ is a critical point. By the condition (4), $\partial U\cap\partial V$ consists of exactly one point $c_0$.
Moreover, $\partial V\setminus\{c_0\}$ contains no critical points since the rest critical points are attracted by $U$. Since $\partial U$ is a Jordan curve and $c_0$ is simple, it follows that $\partial V$ is locally a Jordan arc. Thus $V$ is a Jordan domain.

Let $W$ be a component of $f^{-1}(V)$. Note that $\partial W$ does not contain critical points. It follows that $\partial W$ is a Jordan domain by Lemma \ref{lem:Jordan}(1). Inductively, every component of $\bigcup_{k\geqslant 1}f^{-k}(V)$ does not contain critical points. Thus all Fatou components of $f$ are Jordan domains.
\end{proof}

\begin{proof}[Proof of Theorem \ref{thm:criterion}]
By Lemma \ref{lem:criterion-lc}, it suffices to verify the last three properties in the definition of generalized Sierpi\'{n}ski gaskets.

Note that $F(f)=\bigcup_{k\geqslant 0}f^{-k}(U)$ and $f^{-1}(U)=U\cup V$. There exists a number $D\geqslant 1$ such that for any $k\geqslant 1$ and any component $U'$ of $f^{-k}(V)$, we have
\begin{equation}\label{equ:deg-f-k}
\deg(f^{\circ k}: U'\to V)\leqslant D.
\end{equation}
Let $U_1$ and $U_2$ be two different Fatou components such that $\partial U_1\cap\partial U_2\neq\emptyset$, and $k_1$, $k_2$ be the smallest integers such that $f^{\circ k_1}(U_1)=U=f^{\circ k_2}(U_2)$, where $k_2\geqslant k_1\geqslant 0$.
If $k_2>k_1$, then we have $f^{\circ (k_2-1)}(U_1)=U$ and $f^{\circ (k_2-1)}(U_2)=V$. Since $\partial U\cap\partial V=\{c_0\}$, by \eqref{equ:deg-f-k}, we have $\sharp\big(\partial U_1\cap \partial U_2\big)\leqslant D$.
If $k_2=k_1$, then $f^{\circ (k_1-1)}(U_1)=V=f^{\circ (k_1-1)}(U_2)$ and $k_1\geqslant 2$ since $U_1\neq U_2$. This implies that $\bigcup_{j=0}^{k_1-2}f^{\circ j}(\partial U_1\cup\partial U_2)$ contains a critical point. However, since $c_0$ is simple, there are exactly two Fatou components $U$ and $V$ attaching at $c_0$, which is a contradiction.
Thus $\sharp(\partial U_1\cap\partial U_2)\leqslant D$ for any different components $U_1$, $U_2$ of $F(f)$.

Assume that there are three Fatou components $U_0$, $U_1$ and $U_2$ of $f$ having a common boundary point $z_0$. Since these Fatou components are eventually iterated to $U$ by the iteration of $f$, the forward orbit of $z_0$ must contain at least two critical points, which contradicts to the condition (4).

The rest proof is similar to that of Lemma \ref{lem:gasket-Sie}. For any Fatou component $U_0$ which is different from $U$ and $V$, there exists a smallest integer $k\geqslant 1$ such that $f^{\circ k}(U_0)=V$. Then there exists a Fatou component $U_1$ such that
\begin{equation}
\partial U_0\cap\partial U_1\neq\emptyset \text{\quad and\quad} f^{\circ (k-1)}(U_1)=U.
\end{equation}
Inductively, there exists a sequence of Fatou components $U_0$, $U_1$, $\cdots$, $U_\ell=U$, where $\ell\geqslant 1$ such that
$\partial U_i\cap \partial U_{i+1}\neq\emptyset \text{ for all } 0\leqslant i\leqslant \ell-1$.
Thus the contact graph corresponding to $J(f)$ is connected. The proof is complete.
\end{proof}

For Theorem \ref{thm:criterion}, if $\partial U$ is allowed to contain more critical points, it is not hard to verify that one can still obtain gasket Julia sets by replacing the condition (4) by
\begin{itemize}
\item $\partial U$ contains at least one critical point and all critical points on $\partial U$ are simple and strictly preperiodic;
\item the forward orbit of every critical point on $\partial U$ does not contain any other critical point; and
\item the forward orbit of every critical point outside of $\partial U$ intersects with $U$.
\end{itemize}

\subsection{Specific families}

In this subsection, we study the family $g_\eta$ of rational maps introduced in the introduction and prove that in some cases the Julia sets are generalized Sierpi\'nski gaskets.
Recall that
\begin{equation}
g_\eta(z)=\frac{a(z-1)^3(z-b)^3}{z^4},
\end{equation}
where $a=\frac{\eta(\eta-4)^3}{27(\eta-1)^6}$, $b=\frac{\eta(1+2\eta)}{4-\eta}$ are chosen with $\eta\in\C\setminus\{0,1,4,-\tfrac{1}{2}\}$ such that $g_\eta$ has the critical orbits
\begin{equation}\label{equ:crit-orbit}
b\overset{(3)}{\longmapsto} 0\overset{(4)}{\longmapsto}\infty\overset{(2)}{\longmapsto}\infty\text{\quad and\quad}\eta\overset{(2)}{\longmapsto} 1\overset{(3)}{\longmapsto} 0,
\end{equation}
and $c=\frac{2(1+2\eta)}{\eta-4}$ is a free critical point with local degree $2$.
Let $U_\eta$ be the immediate super-attracting basin of $\infty$. We will use Theorem \ref{thm:criterion} to prove that if $c\in\partial U_\eta$ is strictly preperiodic, then $J(g_\eta)$ is a generalized Sierpi\'{n}ski gasket. See Figure \ref{Fig:Julia-gasket-criterion}.

\begin{figure}[!htpb]
  \setlength{\unitlength}{1mm}
  \centering
  \includegraphics[width=0.47\textwidth]{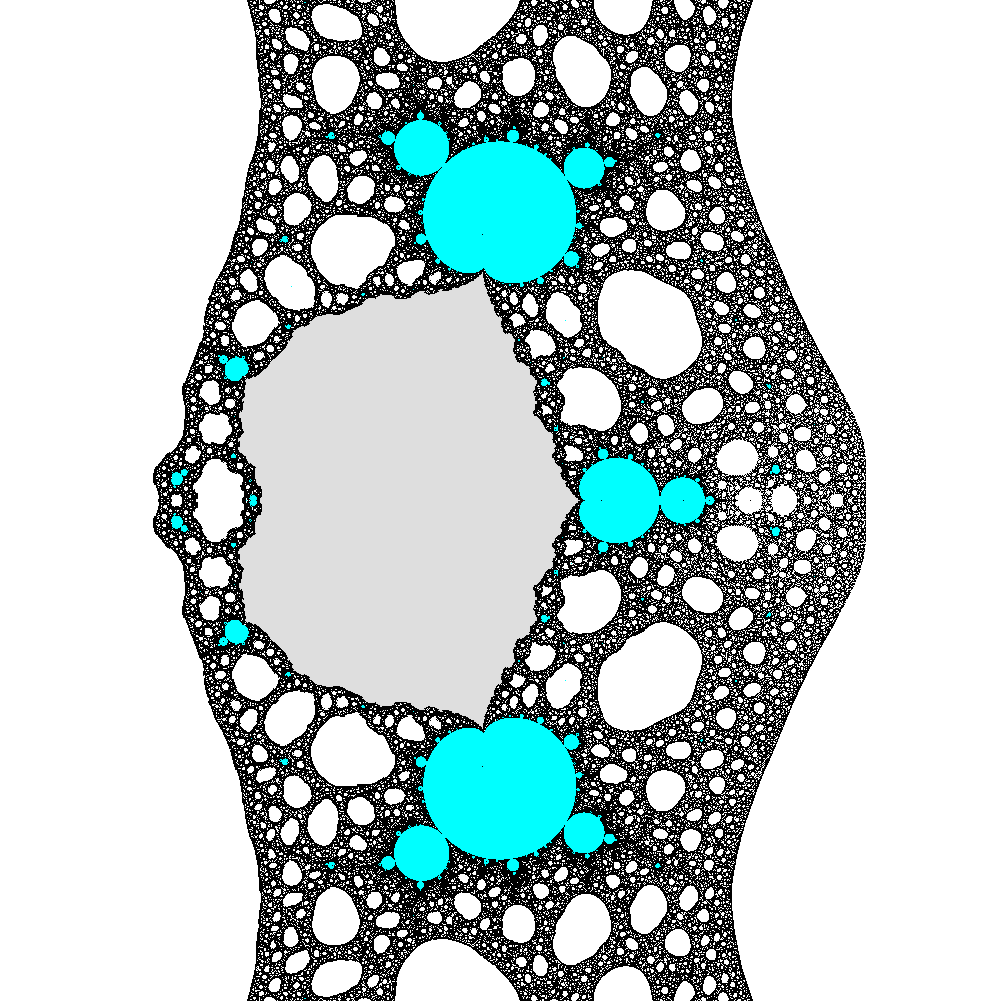}\quad
  \includegraphics[width=0.47\textwidth]{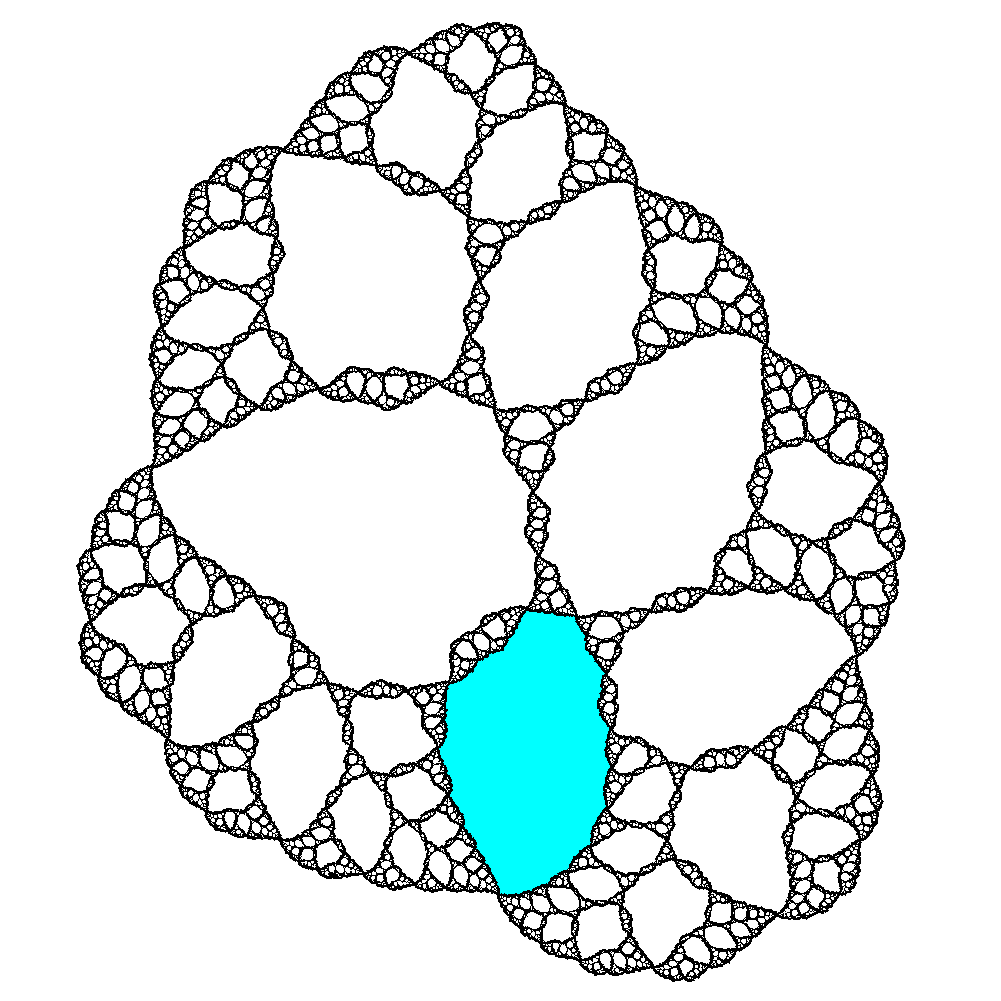}
  \put(-112,32){$\MH_1$}
  \put(-33,14.5){$V_\eta$}
  \put(-53,4){$U_\eta$}
  \put(-36,4){$c$}
  \caption{Left: The parameter plane of $g_\eta$, where the gray domain $\MH_1$ containing a punctured neighborhood of $1$ is the Cantor locus (all critical points are contained in $U_\eta$). Right: The dynamical plane of a specific $g_{\eta}$, where $\eta\approx 0.68771873+ 1.49006605i\in\partial\MH_1$ is chosen such that $c\in\partial U_\eta$, $g_{\eta}^{\circ 3}(c)=g_{\eta}^{\circ 2}(c)$ and $J(g_\eta)$ is a generalized Sierpi\'{n}ski gasket. The immediate preimage $V_\eta$ of $U_\eta$ is marked cyan.}
  \label{Fig:Julia-gasket-criterion}
\end{figure}

\begin{proof}[Proof of Corollary \ref{cor:gasket-2}]
Let $\eta$ be chosen such that the free critical point $c$ lies on $\partial U_\eta$ and is strictly preperiodic. Then $g_\eta$ is critically finite and hence $J(g_\eta)$ is connected (see \cite[p.\,35]{McM88}).

We claim that $0$ and $\infty$ are in different Fatou components. Otherwise, $U_\eta$ is completely invariant. Since $U_\eta$ is simply connected, by the Riemann-Hurwitz formula, $U_\eta$ contains exactly $5$ critical points (counting with multiplicity). However, by \eqref{equ:crit-orbit}, $U_\eta$ contains at least $6$ critical points  (counting with multiplicity), which is a contradiction. Hence $0$ and $\infty$ are in different Fatou components. Let $V_\eta$ be the Fatou component containing $0$. Then we have $g_\eta^{-1}(U_\eta)=U_\eta\cup V_\eta$.

Since $g_\eta^{-1}(0)=\{1,b\}$, by \eqref{equ:crit-orbit}, we have
\begin{equation}
\deg(g_\eta:W_\eta^1\to V_\eta)=\deg(g_\eta:W_\eta^b\to V_\eta)=3>2=\deg(g_\eta:U_\eta\to U_\eta),
\end{equation}
where $W_\eta^1$ and $W_\eta^b$ are the Fatou components containing $1$ and $b$ respectively.

The set of critical points of $g_\eta$ is $\{\infty, 0, 1,b,c,\eta\}$. Except the critical point $c$, the remaining critical points are attracted by $\infty$. Moreover, by the assumption, $c\in\partial U_\eta$ is simple and preperiodic. Therefore, by Theorem \ref{thm:criterion}, the Julia set of $g_\eta$ is a generalized Sierpi\'{n}ski gasket.
\end{proof}

To see the existence of the parameters $\eta$'s satisfying Corollary \ref{cor:gasket-2}, we consider the hyperbolic component $\MH_1$ of $g_\eta$ containing a punctured neighborhood of $1$.
In fact, if $\eta\to 1$, then $a\to \infty$ and $b\to 1$. Hence if $|\eta-1|>0$ is small, then $g_\eta$ can be seen as a small perturbation of $z\mapsto a (z-1)^6/z^4$. Note that the latter lies in the Cantor locus of the varied family of McMullen maps (see \cite{Ste06}). Thus $g_\eta$ lies in the Cantor locus if $|\eta-1|>0$ is small. This confirms the existence of the Cantor locus $\MH_1$ (see the left picture in Figure \ref{Fig:Julia-gasket-criterion}). Moreover, by applying the Riemann-Hurwitz formula, it is not hard to see that $J(g_\eta)$ is connected for all $\eta\in\partial\MH_1$.

Note that $g_\eta$ has exactly one free simple critical point $c$.
By a similar proof of \cite[Section 5]{BDL06}, one can construct a typical conformal isomorphism $\Phi:\MH_1\to\C\setminus\overline{\D}$ and define parameter rays in $\MH_1$.
Similar to \cite{PR00}, one can show that the landing points of some rational parameter rays on $\partial\MH_1$ are \textit{Misiurewicz} parameters (i.e., the free critical point $c$ is strictly preperiodic in the Julia set) and there are corresponding dynamical rays in $U_\eta$'s landing at the critical point $c\in\partial U_\eta$. This indicates the existence of the parameters $\eta$'s satisfying Corollary \ref{cor:gasket-2}.

\medskip
For Theorem \ref{thm:criterion}, the condition (4) requiring that the critical point on $\partial U$ has local degree $2$ \textit{is necessary}.
As a counterexample, inspired by Corollary \ref{cor:gasket-2}, we consider the following one-dimensional family of rational maps:
\begin{equation}
g_a(z)=\frac{a(z-1)^3(z-b)^3}{z^4}, \text{\quad where }
b=-17-12\sqrt{2}
\end{equation}
and $a\in\C\setminus\{0\}$ is the parameter. One can see that $g_a$ has the critical orbits
\begin{equation}
b\overset{(3)}{\longmapsto} 0\overset{(4)}{\longmapsto}\infty\overset{(2)}{\longmapsto}\infty\text{\quad and\quad} 1\overset{(3)}{\longmapsto} 0,
\end{equation}
and $c=4+3\sqrt{2}$ is a free critical point with local degree $3$. Let $U_a$ and $V_a$ be the Fatou components of $g_a$ containing $\infty$ and $0$ respectively.
If $c\in\partial U_a$ is critically finite, then $U_a$ is a Jordan domain (similar to Lemma \ref{lem:criterion-lc}) but $V_a$ is not (by the local mapping property at $c\in\partial U_a\cap \partial V_a$). Hence $J(g_a)$ is not a generalized Sierpi\'{n}ski gasket.





\bibliographystyle{amsalpha}
\bibliography{E:/Latex-model/Ref1}

\end{document}